\documentclass[9pt]{article}
\usepackage{helvet}
\usepackage[all]{xy}
\usepackage{amssymb, amsthm}
\usepackage{epsfig}

\usepackage{amsmath}
\usepackage{mathrsfs}
\usepackage{algorithmic}
\usepackage{algorithm}

\usepackage{rotating}
\usepackage{multirow}
\usepackage{booktabs}

\graphicspath{{./figs/}}

\newtheorem{theorem}{Theorem}

\newenvironment{packed_enum}{
\begin{enumerate}
\topsep=0pt plus 2pt minus 4pt
  \setlength{\itemsep}{1pt}
  \setlength{\parskip}{0pt}
  \setlength{\parsep}{0pt}
}{\end{enumerate}}

\begin{document}

\begin{center}
  \textbf{\LARGE On random coarsening and its applications}\vspace*{3ex}\\
  Pawan Kumar\footnote{This work was done when the author was previously funded by Fonds de la recherche scientifique (FNRS)(Ref: 2011/V 6/5/004-IB/CS-15) at ULB, Brussels and post-doctoral funding at KU Leuven, Belgium}\\
  % Institut Henri Poincare (IHP) - UMS 839 (CNRS/UPMC)\\
  % 11, rue Pierre et Marie Curie - 75231 Paris Cedex 05 \\
  % France \\
  Department of computer science \\
  KU Leuven \\
  Leuven, Belgium \\
  \textsf{pawan.kumar@cs.kuleuven.be}\vspace*{0.5ex}\\
%  {\small Accepted for 10th IMACS symp. 2011, Morocco}
\end{center}
      % title, author, ...
\section*{Abstract}
In this paper, we use the Poincare separation theorem for estimating
the eigenvalues of the fine grid. We propose a randomized version of
the algorithm where several different coarse grids are constructed
thus leading to more comprehensive eigenvalue estimates. The proposed
algorithm is suited for modern day multicore and distributed
processing in the sense that no communication is required between the
processors, however, at the cost of possible redundant computation.

       % abstract
\section{Introduction}
The problem of obtaining an approximation to eigenvalues and
eigenvectors appears in several applications including data mining,
chemical research, vibration analysis of mechanical structures, image
processing etc. On the other hand, singular value decomposition has
many useful applications in signal processing and statistics. For
iterative methods, an estimate of extreme eigenvalue is useful for
rapid Chebychev method \cite{saad96} and in the construction of
deflation preconditioners. An estimate of extreme eigenvalue leads to
an estimate of condition number for symmetric positive matrix.

Poincare separation theorem \cite{rao} states that the eigenvalues of
coarse grid matrix $P^TAP$ are ``sandwiched'' between the eigenvalues
of the fine grid matrix $A$. In this paper, we consider samples of
randomized coarsening scheme, i.e., the fine grid matrix is coarsened
using special randomized interpolation operators $P$ leading to
several samples of coarse grids preferably with different distribution
of eigenvalues. We then compute the eigenvalues of these coarse grid
matrices. When a sufficiently large number of coarse grids are taken
then the smallest eigenvalue (singular value) of the fine grid is
approximated by the smallest of the eigenvalues (singular values) of
the coarse grid matrices and the largest eigenvalue (singular value)
of the fine grid is approximated by the largest of the eigenvalues
(singular values) of the coarse grid matrices.  On the other hand, it
is also possible to use the eigenvalues(singular values) of the coarse
grid matrices as shifts for computing the eigenvalues(singular values)
for the fine grid matrix.

The proposed algorithm is well suited for modern day multi-core and
multiprocessor era since coarsening and subsequently the eigenvalue
(singular value) of the resulting coarse grid could be computed
independently without performing any inter node communication. The
only communication required is when we gather the eigenvalues
(singular values) computed by the processors. Given that communication
often becomes more costly relative to computation it is essential to
degisn algorithms that minimize communication as much as possible even
at the cost of small redundant computation. This is the main reason
behind the method proposed in this paper. However, we do not show any
results for parallel case and here we only focus our study in
understanding the quality of our approach.

The algorithms proposed has some similarity with the Jacobi-Davidson
(JD) method \cite{sle} in the sense that both of these method try to
approach the the eigenvalues of the fine grid via coarse grid,
however, contrary to the sophisticated Jacobi-Davidson method, the
method proposed is based on brute force approach, i.e., the method
relies on creating enough coarse grid samples such that one of these
coarse grid leads to the desired eigenvalue or singular
value. Moreover, unlike JD method where the matrix P keeps growing by
one column during the outer iteration in our method $P$ is fixed thus the
coarse grid matrix $P^TAP$ is also fixed for each coarse grid sample.

This paper is organized as follows. In section (2), we review
essential theorems and motivation behind the algorithms proposed. In
section (3), we explain steps from clustering to obtaining the coarse
matrix. All the algorithms for computing the eigenvalues and
eigenvectors are presented in section (4), here we also show some the results of 
some numerical experiments and finally section (5) concludes this paper.

\section{Poincar\'e separation theorem}

Let $\lambda_i$ denote an arbitrary eigenvalue of $A$.  The trace of
an $n \times n$ matrix $A$ is defined to be the sum of the elements on
the main diagonal of A, i.e.,
\[ tr(A) = \sum_{i=1}^{n}a_{ii}\]. If $f(x) = (x-\lambda_1)^{d_1}
\cdots (x-\lambda_k)^{d_k}$ is the characteristic polynomial of a 
matrix $A$, then tr(A) is defined as follows
\[ tr(A)=d_1\lambda_1+ \cdots + d_k\lambda_k. \] We have the following
relation
\begin{align}
  \sum_{i=1}^{n}a_{ii} = \sum_{i=1}^{n}\lambda_{i}
\end{align}

Let $K^T$ denote the transpose of a matrix $K$ and let $I_k$ denote
the identity matrix of size $k$.  Here we will see how poincar\'e
separates eigenvalues of two grids.

\begin{theorem}[Poincar\'e]\label{Th:poincare}
  Let $A$ be a symmetric $n \times n$ matrix with eigenvalues
  $\lambda_1 \le \lambda_2 \le \dots \le \lambda_n $, and let $P$ be a
  semi-orthogonal $n \times k$ matrix with the property that
  $P^TP=I_k$. The eigenvalues $\mu_1 \le \mu_2 \le \cdots \le
  \mu_{n-k+i}$ of $P^TAP$ are separated by the eigenvalues of $A$ as
  follows
  \begin{align}
    \lambda_i \le \mu_{i} \le \lambda_{n-k+i}.
  \end{align}
\end{theorem}
\begin{proof}
  The theorem is proved in \cite{rao}.
\end{proof}
In Figure (\ref{poincare_sky10_zoom}), we show a part of the spectrum
where eigenvalues of a coarse grid is distributed among the fine grid
eigenvalues.

\begin{figure}
  \caption{Poincare separates for sky 10$\times$10$\times$10 zoomed}
  \label{poincare_sky10_zoom}
  \begin{center}
    \includegraphics[scale=0.3]{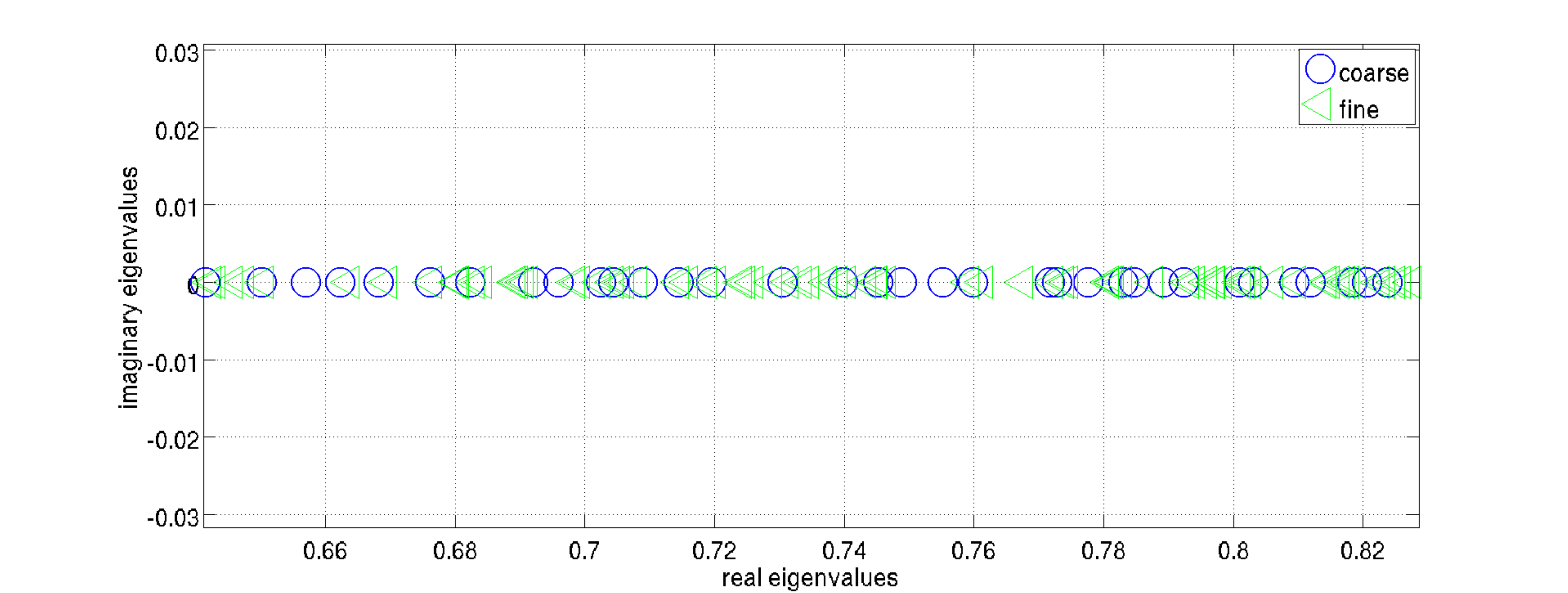}
  \end{center}
\end{figure}

\begin{theorem}\label{Th:poincare_cor1}
  If $A$ is a real symmetric $n \times n$ matrix with eigenvalues
  $\lambda_1 \le \lambda_2 \le \cdots \le \lambda_n$, then the
  following holds
  \begin{align}
    \min_{P^TP=I_k} tr(P^TAP) &= \sum_{i=1}^{k}\lambda_i, \\
    \max_{P^TP=I_k} tr(P^TAP) &= \sum_{i=1}^{k}\lambda_{n-k+i}.
  \end{align}
\end{theorem}
\begin{proof}
  The theorem is proved in \cite{rao}.
\end{proof}
\begin{theorem}
  If $A$ is a real symmetric $n \times n$ matrix with eigenvalues
  $\lambda_1 \le \lambda_2 \le \cdots \le \lambda_n$, and if the
  following conditions are satisfied \vspace{-5mm}
  \begin{packed_enum}
  \item $tr(P^TAP)$ is minimum and
  \item $P^TAP$ has simple eigenvalues
  \end{packed_enum}
  then we have
  \[ \mu_i = \lambda_i, \quad 1 \le i \le k, \] where $\mu_1 \le \mu_2
  \le \cdots \le \mu_{k}$ are eigenvalues of $P^TAP$.
\end{theorem}
\begin{proof}
  Since $P^TAP$ has simple eigenvalues, we have $tr(P^TAP) =
  \sum_{i=1}^{k}\mu_i$. Also, from Theorem \ref{Th:poincare_cor1}
  above, we have
  \begin{align}\label{Eq:trace}
    tr(P^TAP) = \sum_{i=1}^{k}\lambda_i.
  \end{align}
  We shall prove the hypothesis by contradiction. From
  \ref{Th:poincare}, we have $\mu_i \ge \lambda_i$. Let there exist
  $j$, $1 \le j \le k$, such that $\mu_i > \lambda_i$, then get
  \begin{align}
    tr(P^TAP) = \sum_{i=1}^{k}\mu_i > \sum_{i=1}^{k}\lambda_i
  \end{align}
  contradicting \eqref{Eq:trace}. Thus, we must have $\mu_i =
  \lambda_i$. The proof is complete.
\end{proof}
The theorem above tells us that if we are able to find a matrix $P$
such that $tr(P^TAP)$ is minimum, then the first $k$ smallest
eigenvalues of the matrix $A$ are simply the eigenvalues of the matrix
$P^TAP$ provided $P^TAP$ has simple eigenvalues.
\begin{theorem}
  If $A$ is a real symmetric $n \times n$ matrix with eigenvalues
  $\lambda_1 \le \lambda_2 \le \cdots \le \lambda_n$, and if the
  following conditions are satisfied \vspace{-5mm}
  \begin{packed_enum}
  \item $tr(P^TAP)$ is maximum and
  \item $P^TAP$ has simple eigenvalues
  \end{packed_enum}
  then we have
  \[ \mu_i = \lambda_{n-k+i}, \quad 1 \le i \le k, \] where $\mu_1 \le
  \mu_2 \le \cdots \le \mu_{k}$ are eigenvalues of $P^TAP$.
\end{theorem}
\begin{proof}
  Since $P^TAP$ has simple eigenvalues, we have $tr(P^TAP) =
  \sum_{i=1}^{k}\mu_i$. Also, from Theorem \ref{Th:poincare_cor1}
  above, we have
  \begin{align}\label{Eq:trace2}
    tr(P^TAP) = \sum_{i=1}^{k}\lambda_{n+i-k}.
  \end{align}
  We shall prove the hypothesis by contradiction. From
  \ref{Th:poincare}, we have $\mu_i \le \lambda_i$. Let there exist
  $j$, $1 \le j \le k$, such that $\mu_i < \lambda_i$, then we get
  \begin{align}
    tr(P^TAP) = \sum_{i=1}^{k}\mu_i < \sum_{i=1}^{k}\lambda_i
  \end{align}
  contradicting \eqref{Eq:trace2}. Thus, we must have $\mu_i =
  \lambda_i$. The proof is complete.
\end{proof}
The theorem above tells us that if we are able to find a matrix $P$
such that $tr(P^TAP)$ is maximum, then the $k$ largest eigenvalues of
the matrix $A$ are simply the eigenvalues of the matrix $P^TAP$
provided $P^TAP$ has simple eigenvalues. Determining first $k$
smallest or $k$ largest eigenvalues of a matrix is of prime importance
in many applications.

\begin{theorem}[Poincar\'e] Let $A$ be a real $m \times n$ matrix with
  singular values
  \begin{align*}
    \sigma_1(A) \ge \sigma_2(A) \ge \dots
  \end{align*}
  and let $U$ and $V$ be two matrices of order $m \times p$ and $n
  \times q$, respectively, such that $U^*U=I_p$ and $V^*V=I_q$. Let
  $B=U^*AV$ with singular values
  \begin{align*}
    \sigma_1(B) \ge \sigma_2(B) \ge \cdots
  \end{align*}
  then the singular values of $B$ are separated by the singular values
  of $A$ as follows
  \begin{align*}
    \sigma_i(A) \ge \sigma_{i}(B) \ge \sigma_{i+r}(A), i =
    1,2,\cdots,min\{m,n\}
  \end{align*}
  where $r = (m-p) + (n-q)$
\end{theorem}
\begin{proof}
  The theorem is proved in \cite{rao}.
\end{proof}

\section{Clustering to coarsening}
Our aim is to estimate the eigenvalues of the fine grid $A$ via the
eigenvalues of coarse grid ($P^TAP$). Thus, the first step is
clustering which then leads to the interpolation operator $P$ as
follows. First a set of aggregates $G_{i}$ are defined. There are
several different ways of doing aggregation (also described in
\cite{kum}), some of them are as follows:

\begin{itemize}
\item This approach
  is closely related to the classical AMG \cite{not} where one first
  defines the set of nodes $S_i$ to which $i$ is strongly negatively
  coupled, using the Strong/Weak coupling threshold $\beta$:
  \[
  S_i = \{ \, j \neq i \mid a_{ij} < -\beta \ \text{max}|a_{ik}| \, \}.
  \]
  Then an unmarked node $i$ is chosen such that priority is given to
  the node with minimal $M_i$, here $M_i$ being the number of unmarked
  nodes that are strongly negatively coupled to $i$ \cite{not}.
\item Several graph partitioning methods exists. Aggregation for AMG
  is created by calling a graph partitioner with number of aggregates
  as an input. The subgraph being partitioned are considered as
  aggregates. For instance, in this paper we use this approach by
  giving a call to the METIS graph partitioner routine
  METIS\_PartGraphKway with the graph of the matrix and number of
  partitions as input parameters. The partitioning information is
  obtained in the output argument ``part". The part array maps a given
  node to its partition, i.e., part($i$) = $j$ means that the node $i$
  is mapped to the $jth$ partition. In fact, the part array
  essentially determines the interpolation operator $P$. For instance,
  we observe that the ''part`` array is a discrete many to one
  map. Thus, the $i$th aggregate $G_{i}=\text{part}^{-1}(i)$, where
  \[
  \text{part}^{-1}(i) = \{ \, j \in [1,\, N] \enspace \mid \enspace
  \text{part}(j)=i\,\}
  \]
%  Such graph matching techniques were explored in \cite{kim,bra,ras}.
\item K-means clustering (see MATLAB): This clustering is defined in
  MATLAB and it produces random clustering i.e., a random ``part'' array
  defined above.
\end{itemize}
Let $J$
be the number of such aggregates, then the interpolation matrix $P$ is
defined as follows
\begin{equation} \label{interp}
P_{ij} =
\begin{cases}
  1, &\text{if $i \in G_{j}$,}\\
  0, &\text{otherwise,}\\
\end{cases}
\end{equation}
Here, $1 \le i \le N, \, 1 \le j \le J$, $N$ being the size of the
original coefficient matrix $A$.  Let $N=4$ be the size of $A$. Let
there be two aggregates, $G_1 = \{ \, 1, 3 \, \}$ and $G_2 = \{ \, 2,
4 \, \}$, then the restriction operator $P^{T}$ is defined as follows
$P^{T}=\left[\begin{array}{cccc} 1 & 0 & 1 & 0 \\ %\hline
    0 & 1 & 0 & 1
  \end{array}
\right]$.  Further, we assume that the aggregates $G_{i}$ are such
that
\begin{equation} \label{aggr}
G_{i}\cap G_{j}=\phi,~ \text{for}~ i \neq j~ \text{and}~
\cup_iG_{i}= [1,N]
\end{equation}
Here $[1,\, N]$ denotes the set of integers from $1$ to $N$.
 Notice that the matrix $P$ defined above is an $N \times J$ matrix but since it
has only one non-zero entry (which are ``one'') per row, the matrix
can be defined by a single array containing the indices of the non-zero entries. 
The coarse grid matrix $A_c$ may be computed as follows
\[
(A_c)_{ij} = \sum_{k \in G_i} \sum_{l \in G_j} a_{kl}
\]
where $1 \le i, \ j \le N_c$, and $a_{kl}$ is the $(k,l)th$ entry of $A$.

\section{Randomized coarsening and its applications}
In this section, we list the algorithms that may lead to an
approximation of eigenvalues or singular values. In algorithm
(\ref{algo_mult_coarse}), we show the steps for obtaining the
eigenvalues of the input matrix $A$. Here, $\mu^i_j$ denotes the $jth$
eigenvalue of the $ith$ coarse grid. Later in the algorithm at step
(7) $\mu^i_j$ is used as a shift to obtain the eigenvalue of the input
matrix $A$. Since, Poincar\'e separation theorem tells us that
$\mu^i_j$ will lie between two eigenvalues of the input matrix $A$, we
expect it to converge to nearest one. However, it is possible that
some other eigenvalue of other coarse grid also converges to the same
eigenvalue and this redundant computation is inherent in this
approach. In Algorithm (\ref{svd_algo_mult_coarse}), similar algorithm
related to singular values is shown. Notice here that two
interpolation matrices namely $U$ and $V$ are needed. The procedure
for obtaining them is same as for $P$ except that we make use of two
random clustering to construct the coarse grid matrix $U_i^*AV_i$. For
clustering, we use of the ``kmeans'' clustering of MATLAB.

\begin{algorithm}[]
  \caption{Eigenvalue estimate using multiple coarse grid}
  \label{algo_mult_coarse}
  \begin{algorithmic}[1]
    \STATE INPUT: $A$, $J$, $k$ \STATE OUTPUT: $\Lambda=$ eigenvalues of $A$
      
    \FOR{$i$ = 1 to $J$} %\STATE // parallel for
      
    \STATE $A^i_c = P^T_i A P_i$
      
    \STATE Extract $\{\mu^i_1, \mu^i_2, \mu^i_3, \dots \mu^i_k\}$ =
    eigenvalues($A^i_c$)
      
    \FOR{$j$ = 1 to $k$}
    % \STATE // parallel for
    \STATE $\lambda^{i}_j$ = eigenvalues($A$, $\mu^i_j$) // $\mu^i_j$ is
    the shift
    \ENDFOR
    % Here $1$ denotes number of eigenvalues required and $\mu_i$ is
    % the
    % proposed shift
    \ENDFOR
      
    \STATE $\Lambda = \{ \lambda^1_1, \lambda^1_2, \lambda^1_3, \dots,
    \lambda^1_k \} \cup \{\lambda^2_1, \lambda^2_2, \lambda^2_3,
    \dots, \lambda^2_k \} \cup \dots \cup \{ \lambda^J_1, \lambda^J_2,
    \lambda^J_3, \dots, \lambda^J_k\}$
  \end{algorithmic}
\end{algorithm}

\begin{algorithm}[]
  \caption{Singular value estimate using multiple coarse grid}
  \label{svd_algo_mult_coarse}
  \begin{algorithmic}[1]
    \STATE INPUT: $A$, $J$, $k$ \STATE OUTPUT: $\Sigma=$singular value
    estimates of $A$
      
    \FOR{$i$ = 1 to $J$} %\STATE // parallel for
      
    \STATE $A^i_c = U^*_i A V_i$
      
    \STATE Extract $\{\sigma^i_1, \sigma^i_2, \sigma^i_3, \dots
    \sigma^i_k\}$ = eigenvalues($A^i_c$)
      
    \FOR{$j$ = 1 to $k$}
    % \STATE // parallel for
    \STATE $\Sigma^{i}_j(A)$=singularvalue($A$, $\sigma^i_j$) //
    $\sigma^i_j$ is the shift
    \ENDFOR
    % Here $1$ denotes number of eigenvalues required and $\sigma_i$
    % is the
    % proposed shift
    \ENDFOR
      
    \STATE $\Sigma = \{ \Sigma^1_1, \Sigma^1_2, \Sigma^1_3, \dots,
    \Sigma^1_k \} \cup \{\Sigma^2_1, \Sigma^2_2, \Sigma^2_3, \dots,
    \Sigma^2_k \} \cup \dots \cup \{ \Sigma^J_1, \Sigma^J_2,
    \Sigma^J_3, \dots, \Sigma^J_k\}$
  \end{algorithmic}
\end{algorithm}

In Algorithm (\ref{extreme_eig}) and (\ref{extreme_sig}), we present
special cases of the algorithms presented in Algorithms
(\ref{algo_mult_coarse}) and (\ref{svd_algo_mult_coarse}) to compute
extreme eigenvalues and singular values respectively. We simply
extract only the largest and smallest eigenvalues of all coarse grids.
In Figure (\ref{svd_rand}), we plot the singular values for rand(50)
matrix available in MATLAB for 5 coarse grid samples. The coarse grid
eigenvalues are then used as shift to determine the fine grid
eigenvalues. In figure (\ref{svd_rand_zoom}), we see in detail how the
shifts converge to the actual eigenvalues. 

  \begin{figure}
    \caption{Poincar\'e separates for rand(50), $N_c = 22$, J=5}
    \label{svd_rand}
    \begin{center}
      \includegraphics[scale=0.2]{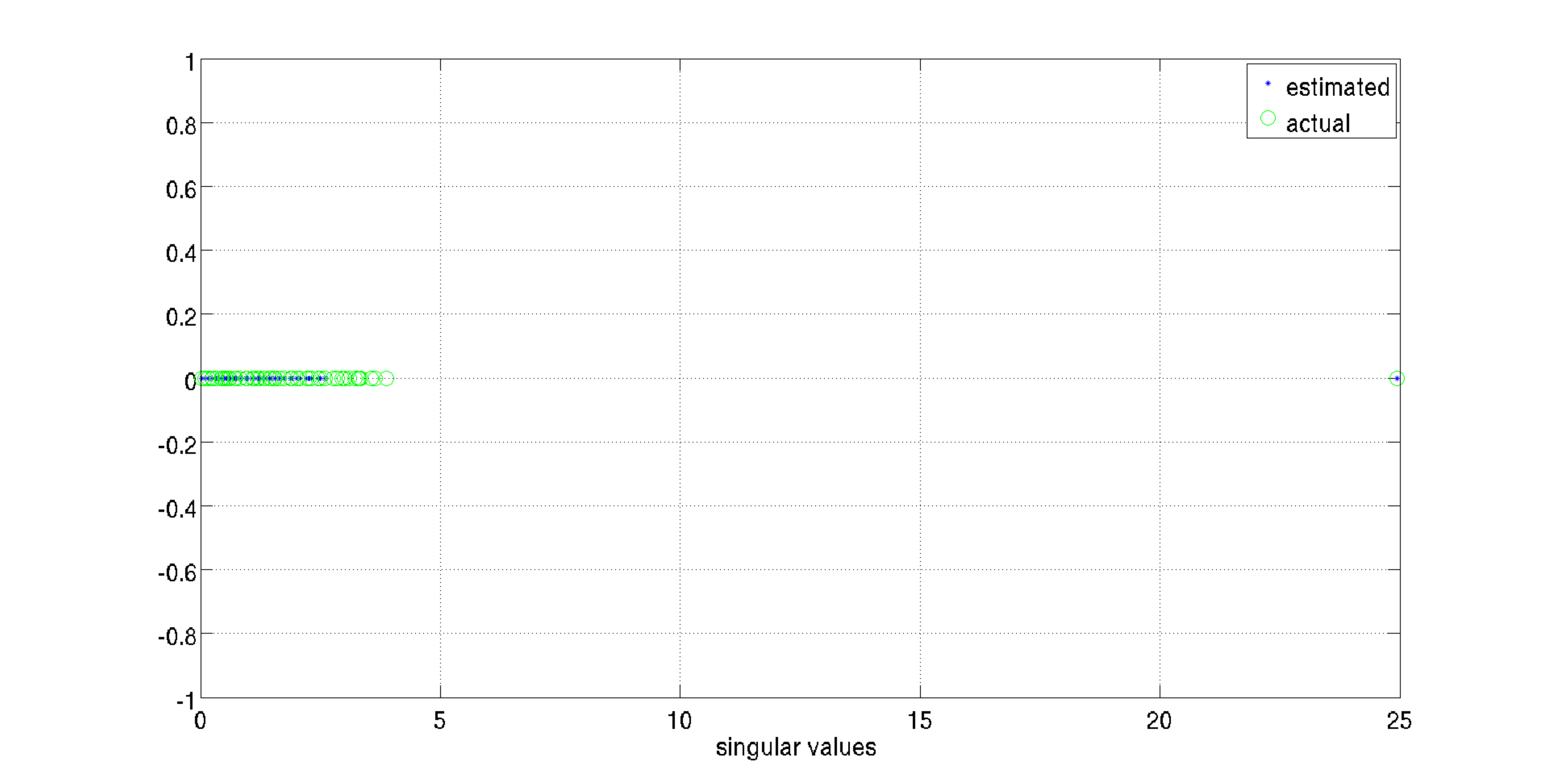}
    \end{center}
  \end{figure}

  \begin{figure}
    \caption{Poincar\'e separates for rand(50), $N_c = 22$, J=5}
    \label{svd_rand_zoom}
    \begin{center}
      \includegraphics[scale=0.2]{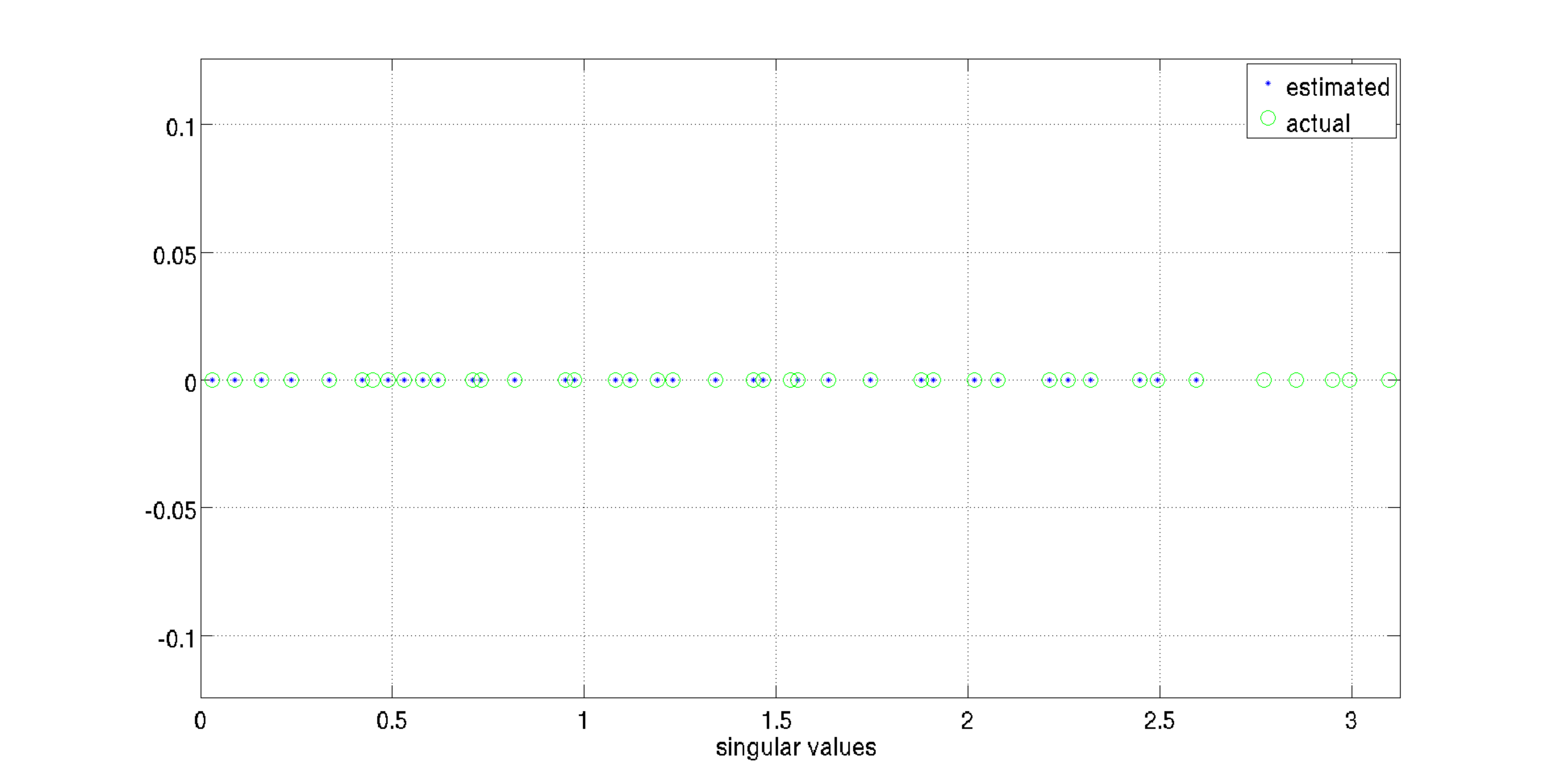}
    \end{center}
  \end{figure}

\begin{algorithm}[]
  \caption{Extreme eigenvalues using multiple coarse grid}
  \label{extreme_eig}
  \begin{algorithmic}[1]
    \STATE INPUT: $A$, $J$ \STATE OUTPUT: $\{\Lambda_{max},
    \Lambda_{min}\}=$ approx. max and min eigenvalues of $A$
      
    \FOR{$i$ = 1 to $J$} %\STATE // parallel for
      
    \STATE $A^i_c = P^T_i A P_i$ // perform coarsening
      
    \STATE $\mu^i_{max}$ = eigmax($A^i_c$) // just find the largest
    eigenvalue \STATE $\mu^i_{min}$ = eigmin($A^i_c$) // just find the
    smallest eigenvalue
    \ENDFOR
      
    \STATE $\Lambda_{max} = max\{ \mu^1_{max}, \mu^2_{max},
    \mu^3_{max}, \dots, \mu^J_{max} \}$ \STATE $\Lambda_{min} = min\{
    \mu^1_{min}, \mu^2_{min}, \mu^3_{min}, \dots, \mu^J_{min} \}$
  \end{algorithmic}
\end{algorithm}

\begin{algorithm}[]
  \caption{Extreme singular values using multiple coarse grid}
  \label{extreme_sig}
    \begin{algorithmic}[1]
      \STATE INPUT: $A$, $J$
      \STATE OUTPUT: $\{\Sigma_{max}, \Sigma_{min}\}=$ approx. max and min singular values of $A$
      
      \FOR{$i$ = 1 to $J$} %\STATE // parallel for
      
      \STATE $A^i_c = U^*_i A V_i$           // perform coarsening
      
      \STATE $\sigma^i_{max}$ = singularmax($A^i_c$) // just find the largest singular value
      \STATE $\sigma^i_{min}$ = singularmin($A^i_c$) // just find the smallest singular value
      \ENDFOR
      
      \STATE $\Sigma_{max} = max\{ \sigma^1_{max}, \sigma^2_{max}, \sigma^3_{max}, \dots, \sigma^J_{max} \}$ 
      \STATE $\Sigma_{min} = min\{ \sigma^1_{min}, \sigma^2_{min}, \sigma^3_{min}, \dots, \sigma^J_{min} \}$ 
    \end{algorithmic}
  \end{algorithm}

       % bibliography

\begin{thebibliography}{99}
\bibitem{kar} G. Karypis, V. Kumar, {\em A fast and high quality
    multilevel scheme for partitioning irregular graphs}, SIAM
  J. Sci. Comp., (1999), 359-392.
  
\bibitem{bell} R. E. Bellman, {\em Introduction to matrix analysis,
    2nd ed.}, New York: McGraw-Hill, p. 117, 1970
  
\bibitem{not} Y. Notay, {\em An aggregation based algebraic
    multigrid}, Num. Lin. Alg. Appl., vol 18, pp 539-564, 2011
  
\bibitem{rao} C. R. Rao and M. B. Rao {\em Matrix algebra and its
    applications to statistics and econometrics}, World scientific
  publishing, 2004

 \bibitem{saad96} Y. Saad, \emph {\it Iterative Methods for Sparse
      Linear Systems}, PWS publishing company, Boston, MA, 1996.

  \bibitem{sle} G. L. G. Sleijpen and H. A. van der Vorst {\em A
      Jacobi-Davidson iteration method for linear eigenvalue
      problems}, SIAM J. Matrix Anal. Appl., vol 17, pp 401-425, 1996.

   \bibitem{kum} P. Kumar, {\em Aggregation based on graph matching and
      inexact coarse grid solve for algebraic multigrid},
    arXiv:1105.3468v5, 2011

  
\end{thebibliography}
\end{document}